\documentclass[12pt,leqno]{article}
\usepackage{amsmath,amssymb,bbm,array,arydshln}
\usepackage{amsthm}
\usepackage{titlesec} % 重設章節標題所需的巨集
\usepackage[all,cmtip]{xy}

\titlespacing{\section}{0cm}{3.5pc}{1.5pc}

\makeatletter
\def\@citex[#1]#2{\if@filesw\immediate\write\@auxout{\string\citation{#2}}\fi
  \def\@citea{}\@cite{\@for\@citeb:=#2\do
    {\@citea\def\@citea{\@citesep}\@ifundefined
       {b@\@citeb}{{\bf ?}\@warning
       {Citation `\@citeb' on page \thepage \space undefined}}%
{\csname b@\@citeb\endcsname}}}{#1}}
\def\@citesep{; }
\makeatother

\newtheoremstyle{Kang}{}{}{\itshape}{}{\bf}{}{.5em}{}
\theoremstyle{Kang}
\newtheorem{theorem}{Theorem}[section]

\newtheorem{lemma}[theorem]{Lemma}

\newtheoremstyle{Kremark}{}{}{}{}{\bf}{}{.5em}{}
\theoremstyle{Kremark}

\newtheorem{defn}[theorem]{Definition}
\newtheorem{example}[theorem]{Example}
\newtheorem{other}{}
\newenvironment{idef}[1]{\renewcommand{\theother}{#1}\begin{other}}{\end{other}}
\newenvironment{Case}[1]{\medskip {\it Case #1.}}{}

\allowdisplaybreaks[1]  % 允許多行數學式分頁

\def\fn#1{\operatorname{#1}} % function work like \sin
\def\bm#1{\mathbbm{#1}}

\title{Frobenius Groups and Retract Rationality}
\author{Ming-chang Kang \\[3mm]
Department of Mathematics and \\ Taida Institute of Mathematical Sciences,\\
National Taiwan University \\ Taipei, Taiwan \\
E-mail: kang@math.ntu.edu.tw}
\date{}

\begin{document}

\maketitle

\footnote{\textit{\!\!\!$2010$ Mathematics Subject
Classification}. Primary 13A50, 14E08, 12F12, 20J06.}
\footnote{\textit{\!\!\!Keywords and phrases}. Noether's problem,
the inverse Galois problem, retract rationality, Frobenius
groups.} \footnote{\textit{\!\!\!}Partially supported by National
Center for Theoretic Sciences (Taipei Office).}

\begin{abstract}
{\noindent\bf Abstract.} Let $k$ be any field, $G$ be a finite
group acting on the rational function field $k(x_g:g\in G)$ by
$h\cdot x_g=x_{hg}$ for any $h,g\in G$. Define $k(G)=k(x_g:g\in
G)^G$. Noether's problem asks whether $k(G)$ is rational (= purely
transcendental) over $k$. A weaker notion, retract rationality
introduced by Saltman, is also very useful for the study of
Noether's problem. We prove that, if $G$ is a Frobenius group with
abelian Frobenius kernel, then $k(G)$ is retract $k$-rational for
any field $k$ satisfying some mild conditions. As an application,
we show that, for any algebraic number field $k$, for any
Frobenius group $G$ with Frobenius complement isomorphic to
$SL_2(\bm{F}_5)$, there is a Galois extension field $K$ over $k$
whose Galois group is isomorphic to $G$, i.e. the inverse Galois
problem is valid for the pair $(G,k)$. The same result is true for
any non-solvable Frobenius group if $k(\zeta_8)$ is a cyclic
extension of $k$.

\end{abstract}

\newpage
%-----------------------------------S1
\section{Introduction}

Let $k$ be any field and $G$ be a finite group acting on the
rational function field $k(x_g:g\in G)$ by $k$-automorphisms so
that $g\cdot x_h=x_{gh}$ for any $g,h\in G$. Denote by $k(G)$ the
fixed field $k(x_g:g\in G)^G$. Noether's problem asks whether
$k(G)$ is rational (= purely transcendental) over $k$. It is
related to the inverse Galois problem, to the existence of generic
$G$-Galois extensions over $k$, and to the existence of versal
$G$-torsors over $k$-rational field extensions \cite[33.1,
p.~86]{Sw,Sa1,GMS}. Noether's problem for abelian groups was
studied by Swan, Voskresenskii, Endo, Miyata and Lenstra, etc. The
reader is referred to Swan's paper for a survey of this problem
\cite{Sw}.

On the other hand, just a handful of results about Noether's
problem for non-abelian groups are obtained.

Before stating the main results of this paper,
we recall a notion introduced by Saltman,
retract rationality, which is weaker than rationality, but is still very useful.

%-------------------d1.1
\begin{defn} \label{d1.1}
Let $k\subset K$ be an extension of fields. $K$ is rational over
$k$ (for short, $k$-rational) if $K$ is purely transcendental over
$k$. $K$ is stably $k$-rational if $K(y_1,\ldots,y_m)$ is rational
over $k$ for some $y_1,\ldots,y_m$ such that $y_1,\ldots,y_m$ are
algebraically independent over $K$.

When $k$ is an infinite field, $K$ is said to be retract
$k$-rational if there is a $k$-algebra $A$ contained in $K$ such
that (i) $K$ is the quotient field of $A$, (ii) there exist a
non-zero polynomial $f\in k[X_1,\ldots,X_n]$ (where
$k[X_1,\ldots,X_n]$ is the polynomial ring) and $k$-algebra
morphisms $\varphi\colon A\to k[X_1,\ldots,X_n][1/f]$ and
$\psi\colon k[X_1,\ldots,X_n][1/f]\to A$ satisfying
$\psi\circ\varphi =1_A$. See \cite{Sa2,Ka4} for details.

It is not difficult to see that ``$k$-rational" $\Rightarrow$
``stably $k$-rational" $\Rightarrow$ ``retract $k$-rational". Note
that $k(G)$ is retract $k$-rational is equivalent to the existence
of a generic $G$-Galois extension over $k$ \cite[Theorem
1.2]{Sa1,Sa2,Ka4}.
\end{defn}

The following result ensures that $k(G)$ is retract $k$-rational
is as good as Noether's problem for solving the inverse Galois
problem. The reader is referred to Serre's monograph \cite{Se} for
Hilbert's irreducibility theorem and the inverse Galois problem.

\medskip
\begin{theorem} \label{t1.12}
Let $G$ be a finite group, and $k$ be an infinite field such that
Hilbert's irreducibility theorem is valid $($e.g. $k$ is an
algebraic number field$)$. If $k(G)$ is retract $k$-rational, then
there is a Galois extension field $K$ over $k$ whose Galois group
is isomorphic to $G$.
\end{theorem}

\begin{proof}
Write $E=k(x_g:g\in G)$ and $F=E^G=k(G)$. Since $k(G)$ is retract
$k$-rational, there is $k$-algebra $A$ in $F$ such that the
quotient field of $A$ is $F$, and the $k$-algebra morphisms
$\varphi\colon A\to k[X_1,\ldots,X_n][1/f]$ and $\psi\colon
k[X_1,\ldots,X_n][1/f]\to A$ satisfying $\psi\circ\varphi =1_A$.

Let $B$ be the integral closure of $A$ in $E$. Write $E=F(\theta)$
for some primitive element $\theta$. We may assume that $\theta
\in B$ and $B$ is a Galois extension of $A$ with group $G$ (see
\cite[Section 3]{Sw} for details).

Since $A[\theta]$ and $B$ have the same quotient field, there
exist elements $s, t \in A$ such that $A[\theta][1/s]=B[1/t]$ by
\cite[Lemma 8]{Sw1}. Further localization will ensure that
$A[\theta]=B$. In short, we have a Galois extension $B$ over $A$
and the $k$-algebra morphisms $\varphi\colon A\to
k[X_1,\ldots,X_n][1/f]$ and $\psi\colon k[X_1,\ldots,X_n][1/f]\to
A$ such that $B=A[\theta]$.

Let $g(T) \in A[T]$ be the minimum polynomial of $\theta$ over $F$
where $T$ is a variable which is algebraic independent over
$k(X_1,\ldots,X_n)$ (and over $F$ via the injective morphism
$\varphi$). Write $g(T)=T^m+a_1T^{m-1}+ \cdots +a_m$ for some
elements $a_1, \cdots, a_m \in A$ with $m=|G|$. Note that $g(T)$
is an irreducible polynomial in $F[T]$. Define
$h(T)=T^m+\varphi(a_1)T^{m-1}+ \cdots +\varphi(a_m)$. We claim
that $h(T) \in k(X_1,\ldots,X_n)[T]$ is an irreducible polynomial
over $k(X_1,\ldots,X_n)$ ; otherwise, applying the morphism $\psi$
would lead to a contradiction.

Now we may apply Hilbert's irreducibility theorem. There is a
$k$-specialization $\lambda: k[X_1,\ldots,X_n][1/f] \to k$ such
that $T^m+\lambda(\varphi(a_1))T^{m-1}+ \cdots
+\lambda(\varphi(a_m))$ is irreducible in $k[T]$. Thus $\lambda
\circ \varphi: A \to k$ is the $k$-specialization we need. The
remaining proof is similar to \cite[Section 3]{Sw}.
\end{proof}

Now we turn to Frobenius groups.

%---------------------d1.5
\begin{defn}[{\cite[pages 181--182]{Is}}] \label{d1.5}
A finite group $G$ is called a Frobenius group if $G=N\rtimes G_0$
where $N$ and $G_0$ are non-trivial subgroups of $G$ satisfying
(i) $N$ is a normal subgroup of $G$, and (ii) the action of $G_0$
on $N$ is fixed point free, i.e.\ for any $x\in N\backslash
\{1\}$, any $g\in G_0\backslash \{1\}$, $gx g^{-1}\ne x$.

In other words, a group $G$ is a Frobenius group if and only if
the celebrated Frobenius Theorem in representation theory is valid
for $G$ \cite[Theorem 8.5.5, page 243]{Ro}. In this situation, the
normal subgroup $N$ is called the Frobenius kernel of $G$, and the
subgroup $G_0$ (or any of its conjugates in $G$) is called a
Frobenius complement of $G$.

A group is called a Frobenius complement (resp. a Frobenius
kernel) if it is a Frobenius complement (resp. a Frobenius kernel)
of some Frobenius group.
\end{defn}

%----------------------t1.6
\begin{theorem}[{\cite[page 299; Is, Chapter 6]{Ro}}] \label{t1.6}
Let $G=N\rtimes G_0$ be a Frobenius group with kernel $N$ and
complement $G_0$.

{\rm (1)} $(John \, Thompson)$ $N$ is a nilpotent group.

{\rm (2)} $(Burnside)$ The $p$-Sylow subgroups of $G_0$ are cyclic
if $p\ge 3$, while the $2$-Sylow subgroups of $G_0$ are cyclic or
generalized quaternion.
\end{theorem}

Frobenius groups and Frobenius complements are ubiquitous in
various mathematical areas (see Section 2). We will coin some
names for these groups.

%------------------d2.2
\begin{defn} \label{d2.2}
Following the terminology used by Suzuki \cite{Su}, a finite group
$G$ is called a $Z$-group if all of its Sylow subgroups are
cyclic. Imitating this usage, we called a finite $G$ a $GZ$-group
if the $p$-Sylow subgroups of $G$ are cyclic for $p\ge 3$ and the
2-Sylow subgroups of $G$ are cyclic or generalized quaternion.

From Theorem \ref{t1.6}, we find that the Frobenius complements
are $GZ$-groups.
\end{defn}

The purpose of this paper is to study the rationality problems of
Frobenius groups and Frobenius complements. But let us discuss an
example first.

%-----------------------ex1.12
\begin{example} \label{ex1.12}
Let $p$, $q$ be distinct odd prime numbers. Define $G=\langle
\sigma,\tau: \sigma^p=\tau^q=1,~
\tau\sigma\tau^{-1}=\sigma^r\rangle$ where $2\le r\le p \, - \,1$
and $r^q\equiv 1$ (mod $p$). Then $G$ is a Frobenius group and
$G\simeq C_p \rtimes C_q$. By Theorem \ref{t3.4}, if
$\bm{Z}[\zeta_q]$ is a unique factorization domain, then
$\bm{C}(C_p\rtimes C_q)$ is rational over $\bm{C}$. However, when
$\bm{Z}[\zeta_q]$ is not a unique factorization domain, it is
still unknown whether $\bm{C}(C_p\rtimes C_q)$ is rational or not.
\end{example}

Because of the above example, we will study the retract
rationality of $k(G)$ where $k$ is a field satisfying mild
assumptions, and $G$ is a $GZ$-group, a Frobenius complement or a
Frobenius group with abelian kernel. We will not aim at the
rationality problem for these groups in this paper. The
rationality problem of $\bm{Q}(G)$ where
$G=(\bm{Z}/p\bm{Z})\rtimes (\bm{Z}/p\bm{Z})^{\times}$ was indeed
explored by Samson Breuer \cite{Bre} (when $p$ is a prime number
with $p\le 19$); we will study the rationality problem for these
groups in a separate paper.

The retract rationality for a Frobenius group with abelian kernel
is a consequence of that for $GZ$-groups (for the details about
the results for $GZ$-groups, see Section 4). We list our results
for Frobenius groups as follows.

%----------------------t1.9
\begin{theorem} \label{t1.9}
Let $G$ be a solvable Frobenius group of exponent $e=2^u3^ln$
$($where $u,l\ge 0$, $2\nmid n$, $3\nmid n)$, and $k$ be an
infinite field with $\fn{char}k\ne 2,3$, and $\zeta_{2^{u'}},
\zeta_{3^l} \in k$ where $u'=\max\{u,3\}$. If the Frobenius kernel
of $G$ is an abelian group, then $k(G)$ is retract $k$-rational.
\end{theorem}

%-----------------------t1.8
\begin{theorem} \label{t1.8}
Let $G$ be a non-solvable Frobenius group. If $k$ is an infinite
field with $\fn{char}k=2$ or $\fn{char}k=0$ such that $k(\zeta_8)$
is a cyclic extension of $k$, then $k(G)$ is retract $k$-rational.
\end{theorem}

Note that, in Theorem \ref{t1.8}, the assumption that the
Frobenius kernel is abelian is unnecessary because of Theorem
\ref{t2.1} and Theorem \ref{t2.9}. A criterion of retract
rationality when the Frobenius complement is a $Z$-group and the
Frobenius kernel is abelian is given in Theorem \ref{c4.2}.

\begin{example} \label{ex1.16}
Let $p$ be a prime number with $p \equiv 1$ (mod $8$). Choose an
integer $r$ such that $2\le r\le p \, - \,1$ and the order of $r$
in $(\bm{Z}/p\bm{Z})^{\times}$ is $8$. Define $G=\langle
\sigma,\tau: \sigma^p=\tau^8=1,~
\tau\sigma\tau^{-1}=\sigma^r\rangle$. Then $G$ is a Frobenius
group and $G\simeq C_p \rtimes C_8$. We claim that $\bm{Q}(G)$ is
not retract $\bm{Q}$-rational. Otherwise, $\bm{Q}(C_8)$ would be
retract $\bm{Q}$-rational by Theorem \ref{t1.11}. But this is
impossible by Theorem \ref{t3.1}.

The above example illustrates that, in Theorem \ref{t1.9}, the
assumption that $\zeta_8 \in k$ is crucial. A similar situation
happens to Theorem \ref{t1.8}; see Example \ref{ex4.13}.
\end{example}

Applying Theorem \ref{t1.9} and Theorem \ref{t1.8} together with
Theorem \ref{t1.12}, we deduce results of the inverse Galois
problem. We record only results for non-solvable Frobenius groups.

\medskip
\begin{theorem} \label{t1.13}
Let $G$ a non-solvable Frobenius group. If $k$ is an algebraic
number field such that $k(\zeta_8)$ is a cyclic extension of $k$,
then there is a Galois extension field $K$ over $k$ whose Galois
group is isomorphic to $G$.
\end{theorem}

According to Suzuki's classification (see Theorem \ref{t2.8}),
there are two types of non-solvable Frobenius groups. The result
in Theorem \ref{t1.8} for the type (I) group in Theorem \ref{t2.8}
may be sharpened as follows.

\medskip
%--------------------------t4.14
\begin{theorem} \label{t1.14}
Let $k$ be a field with $\fn{char}k=0$, $G$ be a Frobenius group
whose Frobenius complement is isomorphic to $G_1\times G_2$ where
$G_1$ is a $Z$-group and $G_2\simeq SL_2(\bm{F}_5)$. Then $k(G)$
is retract $k$-rational.
\end{theorem}

A corollary of the above theorem is the following.

%--------------------------t4.14
\begin{theorem} \label{t1.15}
Let $k$ be an algebraic number field, $G$ be a Frobenius group
whose Frobenius complement is isomorphic to $G_1\times G_2$ where
$G_1$ is a $Z$-group and $G_2\simeq SL_2(\bm{F}_5)$. Then there is
a Galois extension field $K$ over $k$ whose Galois group is
isomorphic to $G$.
\end{theorem}

We note that it is possible to solve the inverse Galois problem in
Theorem \ref{t1.15} by other methods using \cite[page 55, Theorem
3.12]{Fe,Me,ILF} or \cite[page 326, Theorem 8.1]{MM}.

\bigskip
The proof of Theorem \ref{t1.9}, Theorem \ref{t1.8} and Theorem
\ref{t1.14} will be given at the end of Section 4. We use
Saltman's result (Theorem \ref{t1.11}) and the structure theorems
of Frobenius groups as a method of reduction process. The
structure theorems of Frobenius groups were investigated by
Zassenhaus, Suzuki, Wolf and other people \cite{Za1,Su,Wo}. We
summarize these results in Section 2. However, to prove the
retract rationality is rather tricky. Sometimes we should prove
the stronger result of rationality in order to prove the retract
rationality; thus we are reduced to solving Noether's problem for
the corresponding groups. For example, Theorem \ref{t4.5} proves
that $k(G_1)$ and $k(G_2)$ are $k$-rational where $G_1 \simeq Q_8
\rtimes C_{3^l}$, $G_2$ is defined by the group extension $1 \to
G_1 \to G_2 \to C_2 \to 1$ ($Q_8$ is the quaternion group of order
$8$) and $k$ is a field containing $\zeta_e$ with exp$(G_2)=e$.
Also see the proof of Theorem \ref{t4.11} and Theorem \ref{t4.12}.
In Section 5, some remarks about the unramified Brauer groups for
Frobenius groups will be given; here we don't assume the Frobenius
kernels are abelian.

\bigskip
Acknowledgments. I should like to thank I.\ M.\ Isaacs, E.\
Khukhro and Victor Mazurov for providing many helpful messages
about Frobenius groups.

\begin{idef}{Standing notations.}
In discussing retract rationality, we always assume that the
ground field is infinite (see Definition \ref{d1.1}). For
emphasis, recall $k(G)=k(x_g:g\in G)^G$, which is defined in the
first paragraph of this section.

We denote by $\zeta_n$ a primitive $n$-th root of unity in some extension field of the ground field $k$.
When we write $\zeta_n\in k$ or $\fn{char}k\nmid n$,
it is understood that either $\fn{char}k=0$ or $\fn{char}k=p>0$ with $p\nmid n$.

All the groups in this paper are finite groups. $C_n$ denotes the
cyclic group of order $n$. The exponent of a group $G$, $\exp
(G)$, is defined as $\exp(G)=\fn{lcm}\{\fn{ord}(g):g\in G\}$. We
denote $\bm{F}_p$ the finite field with $p$ elements.
\end{idef}

%-----------------------------------S2
\section{Structure theorems of Frobenius groups}

We recall some standard results of Frobenius groups in this section.

%--------------------t2.1
\begin{theorem}[{\cite[Lemma 6.1 and Theorem 6.3]{Is}}] \label{t2.1}
Let $G=N\rtimes G_0$ be a Frobenius group with kernel $N$ and
complement $G_0$. Then

{\rm (1)} $\gcd\{|N|,|G_0|\}=1$, and

{\rm (2)} if $|G_0|$ is even, then $N$ is abelian.
\end{theorem}

%-----------------d2.3
\begin{defn}[{\cite[page 160]{Wo}}] \label{d2.3}
Let $p$ and $q$ be prime numbers (the situation $p=q$ is not excluded).
We say that a finite group $G$ satisfies the $pq$-condition if every subgroup of order $pq$ in $G$ is cyclic.
\end{defn}

%------------------t2.4
\begin{theorem}[{\cite[Theorem 5.3.2, page 161; CE, Theorem 11.6, page 262]{Wo}}] \label{t2.4}
Let $G$ be a finite group.
Then the following conditions are equivalent,
\begin{enumerate}
\item[{\rm (i)}] $G$ satisfies the $p^2$-condition for all prime
numbers $p$; \item[{\rm (ii)}] $G$ is a $GZ$-group; \item[{\rm
(iii)}] every abelian subgroup of $G$ is cyclic; \item[{\rm (iv)}]
$G$ has periodic cohomology, i.e.\ there is some integer $d\ne 0$,
some element $u\in \widehat{H}^d (G,\bm{Z})$ such that the cup
product map $u\cup -:\widehat{H}^n (G,\bm{Z}) \to
\widehat{H}^{n+d} (G,\bm{Z})$ is an isomorphism for all integers
$n$ $($where $\widehat{H}^n (G,\bm{Z})$ is the Tate cohomology$)$.
\end{enumerate}
\end{theorem}

Note that the equivalence of (ii), (iii) and (iv) of the above
theorem was due to Artin and Tate \cite[p.232; Mi, p.627; Su,
p.688]{CE}. The $GZ$-groups arise naturally in algebraic topology
and differential geometry (\cite[pages 357--358; Wo; Jo]{CE}). For
example, if $G$ is a finite group acting on a homology sphere
without fixed points, then (i) P.\ A.\ Smith shows that all
abelian groups of $G$ are cyclic (and therefore $G$ is a
$GZ$-group by Theorem \ref{t2.4}), and (ii) Milnor shows that
there is at most one element of order 2, which is necessary to lie
in the center of $G$ \cite{Mi}. The complete determination of such
groups was solved by Madsen, Thomas and Wall in 1976 (see
\cite[page 158]{Br}).

%-------------------------t3.3
\begin{theorem}[{Burnside \cite[page 281, 10.1.10]{Ro}}] \label{t3.3}
Let $G$ be a finite group. Then $G$ is a $Z$-group if and only if
it is a split meta-cyclic group, i.e. $G=\langle
\sigma,\tau:\sigma^m=\tau^n,~\tau\sigma\tau^{-1}=\sigma^r\rangle$
where $m,n\ge 1$, $r^n\equiv 1 \pmod{m}$, $\gcd\{m,$ $n(r-1)\}=1$.
\end{theorem}

In the above theorem, note that cyclic groups arise when $m=1$.

%--------------------t2.5
\begin{theorem}[{Burnside \cite[Theorem 6.9, page 188]{Is}}] \label{t2.5}
If $G$ is a Frobenius complement, then $G$ satisfies the
$pq$-condition for all prime numbers $p$ and $q$ $($the
possibility $p=q$ is allowed$)$.
\end{theorem}

%--------------------t2.6
\begin{theorem}[{Zassenhaus \cite[Theorem 6.3.1, page 195]{Za1,Za2,Maz,Me,Wo}}] \label{t2.6}
Let $G$ be a Frobenius complement.
If $G$ is a perfect group $($i.e.\ $G$ satisfies the condition $G=[G,G])$,
then $G$ is isomorphic to $SL_2(\bm{F}_5)$.
\end{theorem}

%---------------------t2.7
\begin{theorem}[{Zassenhaus \cite[page 179, Theorem 6.1.11; Jo, pages 164--166]{Za1,Wo}}] \label{t2.7}
Let $G$ be a finite solvable group.
Then $G$ is a $GZ$-group if and only if $G$ is isomorphic to one of the following groups.
\begin{enumerate}
\item[{\rm (I)}] $G=\langle
\sigma,\tau:\sigma^m=\tau^n,~\tau\sigma\tau^{-1}=\sigma^r\rangle$
where $m,n\ge 1$, $r^n\equiv 1 \pmod{m}$, $\gcd\{m,$ $n(r-1)\}=1$;
\item[{\rm (II)}] $G=\langle
\sigma,\tau,\lambda:\sigma^m=\tau^n=1,~\lambda^2=\tau^{n/2},~\tau\sigma\tau^{-1}=\sigma^r,~
\lambda\sigma\lambda^{-1}=\sigma^l,~\lambda\tau\lambda^{-1}=\tau^k\rangle$
where $m,n\ge 1$, $n=2^uv$ with $u\ge 2$, $2\nmid v$, $r^n\equiv 1
\pmod{m}$, $\gcd\{m,n(r-1)\}=1$, $l^2\equiv r^{k-1}\equiv
1\pmod{m}$, $k\equiv -1 \pmod{2^u}$, $k^2\equiv 1 \pmod{n}$;
\item[{\rm (III)}] $G=\langle
\sigma,\tau,\lambda,\rho:\sigma^m=\tau^n=\lambda^4=1,~\lambda^2=\rho^2=(\lambda\rho)^2,~\tau\sigma\tau^{-1}=\sigma^r,~
\lambda\sigma\lambda^{-1}=\sigma,~\rho\sigma\rho^{-1}=\sigma,~\tau\lambda\tau^{-1}=\rho,~\tau\rho\tau^{-1}=\lambda\rho\rangle$
where $m,n\ge 1$, $r^n\equiv 1 \pmod{m}$, $\gcd\{m,n(r-1)\}=1$,
$n\equiv 1 \pmod{2}$ and $n\equiv 0 \pmod{3}$; \item[{\rm (IV)}]
$G=\langle \sigma,\tau,\lambda,\rho,\nu:
\sigma^m=\tau^n=\lambda^4=1,~\lambda^2=\rho^2=(\lambda\rho)^2=\nu^2,~\tau\sigma\tau^{-1}=\sigma^r,~
\lambda\sigma\lambda^{-1}=\sigma,~\rho\sigma\rho^{-1}=\sigma,~\tau\lambda\tau^{-1}=\rho,~\tau\rho\tau^{-1}=\lambda\rho,~
\nu\lambda\nu^{-1}=\rho\lambda,~\nu\rho\nu^{-1}=\rho^{-1},~\nu\sigma\nu^{-1}=\sigma^t,~\nu\tau\nu^{-1}=\tau^k\rangle$
where $m,n\ge 1$, $r^n\equiv r^{k-1}\equiv t^2\equiv 1 \pmod{m}$,
$\gcd\{m,n(r-1)\}=1$, $n\equiv 1 \pmod{2}$, $n\equiv k+1\equiv
0\pmod{3}$, $k^2\equiv 1\pmod{n}$.
\end{enumerate}
\end{theorem}

%-------------------------t2.8
\begin{theorem}[{Suzuki \cite[Theorem E; Wo, pages 197--198; Jo, pages 169--170]{Su}}] \label{t2.8}
Let $G$ be a finite non-solvable group.
Then $G$ is a $GZ$-group if and only if $G$ is isomorphic to one of the following groups,
\begin{enumerate}
\item[{\rm (I)}]
$G=H\times SL_2(\bm{F}_p)$ where $H$ is a $Z$-group, $p$ is a prime number $\ge 5$ and $\gcd\{|H|$, $|SL_2(\bm{F}_p)|\}=1$;
\item[{\rm (II)}]
$G=\langle \sigma,\tau,\lambda,L\rangle$ where $L$ is isomorphic to $SL_2(\bm{F}_p)$ with $p$ being a prime number $\ge 5$ and with the relations
\begin{gather*}
\sigma^m=\tau^n=\lambda^4=1,~ \tau\sigma\tau^{-1}=\sigma^r,~ \lambda\sigma\lambda^{-1}=\sigma^{-1},~
\tau\lambda=\lambda\tau,~ \lambda^2=\varepsilon\in L, \\
\forall \rho\in L,~ \sigma\rho=\rho\sigma,~\tau\rho=\rho\tau,~\lambda\rho\lambda^{-1}=\theta(\rho)
\end{gather*}
where $\varepsilon\in L$ is the element corresponding to $\left(\begin{smallmatrix} -1 & 0 \\ 0 & -1 \end{smallmatrix}\right)\in SL_2(\bm{F}_p)$
under the isomorphism $L\simeq SL_2(\bm{F}_p)$ and $\theta$ is the automorphism on $L$ induced from the automorphism $\theta_0$ on $SL_2(\bm{F}_p)$ defined by
\begin{align*}
\theta_0:{} & SL_2(\bm{F}_p) \to SL_2(\bm{F}_p) \\
& \rho \mapsto \begin{pmatrix} 0 & -1 \\ \omega & 0 \end{pmatrix}
\rho
\begin{pmatrix} 0 & -1 \\ \omega & 0 \end{pmatrix}^{-1}
\end{align*}
where $\omega$ is a generator of the multiplicative group
$\bm{F}_p\backslash \{0\}$ and $r^n\equiv 1 \pmod{m}$,
$\gcd\{m,n(r-1)\}=\gcd\{mn,p(p^2-1)\}=1$.
\end{enumerate}
\end{theorem}

%--------------------------t2.9
\begin{theorem} \label{t2.9}
Let $G_0$ be a non-solvable Frobenius complement. Then $G_0$ is
isomorphic to one of the groups $($I$)$ or $($II$)$ in Theorem
\ref{t2.8} with $p=5$.
\end{theorem}

\begin{proof}
By Theorem \ref{t1.6}, $G_0$ is a non-solvable $GZ$-group.
Hence we may apply Theorem \ref{t2.8}.
It remains to show that $p=5$.

Since $L$ is a subgroup of $G_0$, $L$ is also a Frobenius
complement (reason: If $N\rtimes G_0$ is a Frobenius group, then
$N\rtimes L$ is also a Frobenius group by definition.).

Since $L\simeq SL_2(\bm{F}_p)$ is a perfect group (see, for example, \cite[page 74, 3.2.13]{Ro}),
by Theorem \ref{t2.6}, it is necessary that $p=5$.
\end{proof}

\medskip
Frobenius complements appear as the groups satisfying all the $pq$
conditions in the list of the above Theorem \ref{t2.7}, Theorem
\ref{t2.8} and Theorem \ref{t2.9}. Victor Mazurov informed us many
of his works about Frobenius groups; most of them were written in
Russian under the name V. D. Mazurov. For example, if $G$ is a
sovable group satisfying all the $pq$ conditions in the list of
Theorem \ref{t2.7}, then $G$ is a Frobenius complement. It is
known that $SL_2(\bm{F}_5)$ is a Frobenius complement \cite[page
500]{Hu}; Mazurov had another proof of it. However, it is still
unknown whether non-solvable groups in the list of Theorem
\ref{t2.8} and Theorem \ref{t2.9} satisfying all the $pq$
conditions are eligible Frobenius complements. The following
result is implicit in Zassenhaus's paper \cite{Za1} and is
contained in one of Mazurov's Russian papers.

\medskip
\begin{theorem} \label{t2.10}
Let $G$ be a finite group. Then $G$ is a Frobenius complement if
and only if the subgroup of $G$ generated by all elements of
$($various$)$ prime orders is isomorphic to $C_n \times H$ where
$n$ is a square-free integer, $\gcd\{n, |H| \}=1$ with $H=\{ 1 \},
SL_2(\bm{F}_3)$, or $SL_2(\bm{F}_5)$.
\end{theorem}

%------------------------------------------S3
\section{Preliminaries}

We recall several known results of rationality problems in this section,
which will be used later.

\begin{theorem}[{Saltman \cite[Theorem 3.1 and Theorem 3.5; Ka4, Theorem 3.5]{Sa1}}] \label{t1.11}
Let $k$ be an infinite field, $G=N\rtimes G_0$ where $N$ is a
normal subgroup of $G$ with $G_0$ acting on $N$.

{\rm (1)} If $k(G)$ is retract $k$-rational, so is $k(G_0)$.

{\rm (2)} Assume furthermore that $N$ is abelian and
$\gcd\{|N|,|G_0|\}=1$. If both $k(N)$ and $k(G_0)$ are retract
$k$-rational, so is $k(G)$.
\end{theorem}

%--------------------------t3.1
\begin{theorem}[{\cite[Theorem 4.12; Ka4, Theorem 3.7]{Sa2}}] \label{t3.1}
Let $k$ be an infinite field and $G$ be a finite abelian group of
exponent $e=2^r s$ with $2\nmid s$. Then $k(G)$ is retract
$k$-rational if and only if $\fn{char}k=2$ or $k(\zeta_{2^r})$ is
a cyclic extension of $k$.
\end{theorem}

%-------------------------t3.2
\begin{theorem} \label{t3.2}
Let $G_1$ and $G_2$ be finite groups.
\begin{enumerate}
\item[{\rm (1)}] {\rm (Saltman \cite[Theorem 1.5]{Sa1})}
If $k$ is an infinite field and both $k(G_1)$ and $k(G_2)$ are retract $k$-rational,
then $k(G_1\times G_2)$ is also retract $k$-rational.
\item[{\rm (2)}] {\rm (Kang and Plans \cite[Theorem 1.3]{KP})}
For any field $k$, if both $k(G_1)$ and $k(G_2)$ are $k$-rational,
then $k(G_1\times G_2)$ is also $k$-rational.
\end{enumerate}
\end{theorem}

%-----------------------t3.5
\begin{theorem}[{Ahmad, Hajja and Kang \cite[Theorem 3.1]{AHK}}] \label{t3.5}
Let $L$ be any field, $L(x)$ be the rational function field in one
variable over $L$, and $G$ be a finite group acting on $L(x)$.
Suppose that, for any $\sigma\in G$, $\sigma(L)\subset L$ and
$\sigma(x)=a_\sigma \cdot x+b_\sigma$ where $a_\sigma, b_\sigma
\in L$ and $a_\sigma \ne 0$. Then $L(x)^G=L^G(f)$ for some
polynomial $f\in L[x]$. In fact, if $m=\min \{\fn{deg}
g(x):g(x)\in L[x]^G\backslash L^G\}$, any polynomial $f\in L[x]^G$
with $\fn{deg} f=m$ satisfies the property $L(x)^G=L^G(f)$.
\end{theorem}

%-----------------------t3.6
\begin{theorem}[{Hajja and Kang \cite[Theorem 1]{HK}}] \label{t3.6}
Let $G$ be a finite group acting on $L(x_1,\ldots,x_n)$,
the rational function field in $n$ variables over a field $L$.
Suppose that
\begin{enumerate}
\item[{\rm (i)}]
for any $\sigma \in G$, $\sigma(L)\subset L$,
\item[{\rm (ii)}]
the restriction of the action of $G$ to $L$ is faithful,
\item[{\rm (iii)}]
for any $\sigma\in G$,
\[
\begin{pmatrix} \sigma(x_1) \\ \sigma(x_2) \\ \vdots \\ \sigma(x_n) \end{pmatrix}
=A(\sigma) \cdot \begin{pmatrix} x_1 \\ x_2 \\ \vdots \\ x_n \end{pmatrix}+B(\sigma)
\]
where $A(\sigma)\in GL_n(L)$ and $B(\sigma)$ is an $n\times 1$ matrix over $L$.
\end{enumerate}

Then there exist elements $z_1,\ldots,z_n\in L(x_1,\ldots,x_n)$ so that $L(x_1,\ldots,x_n)=L(z_1$, $\ldots,z_n)$ and $\sigma(z_i)=z_i$ for any $\sigma\in G$, any $1\le i\le n$.
\end{theorem}

%----------------------t3.8
\begin{theorem}[{Kang and Plans \cite[Theorem 1.1]{KP}}] \label{t3.8}
Let $k$ be a field with $\fn{char}k=p>0$ and $\widetilde{G}$ be a
group defined by $1\to \bm{Z}/p\bm{Z} \to \widetilde{G}\to G \to
1$ where $G$ is a finite group. Then $k(\widetilde{G})$ is
rational over $k(G)$. In particular, if $G$ is a $p$-group and $k$
is a field with $\fn{char} k=p>0$, then $k(G)$ is $k$-rational
$($Kuniyoshi's Theorem$)$.
\end{theorem}

%---------------------------t3.7
\begin{theorem}[{Fischer \cite[Theorem 6.1]{Sw}}] \label{t3.7}
Let $G$ be a finite abelian group of exponent $e$, and let $k$ be
a field containing a primitive $e$-th root of unity. Then $k(G)$
is rational over $k$.
\end{theorem}

%-----------------------t3.9
\begin{theorem}[Hajja \cite{Ha}] \label{t3.9}
Let $G$ be a finite group acting on the rational function field $k(x,y)$ by monomial $k$-automorphisms,
i.e.\ for any $\sigma\in G$, $\sigma(x)=\alpha\cdot x^a y^b$, $\sigma(y)=\beta\cdot x^c y^d$ where $\alpha,\beta \in k\backslash \{0\}$,
$a,b,c,d\in \bm{Z}$ and $\alpha$, $\beta$, $a$, $b$, $c$, $d$ are parameters depending on $\sigma$.
Then $k(x,y)^G$ is $k$-rational.
\end{theorem}

\begin{theorem}[{Kang \cite[Corollary 3.2]{Ka2}}] \label{t3.4}
Let $k$ be a field and $G$ be a finite group. Assume that {\rm
(i)} $G$ contains an abelian normal subgroup $H$ so that $G/H$ is
cyclic of order $n$, {\rm (ii)} $\bm{Z}[\zeta_n]$ is a unique
factorization domain, and {\rm (iii)} $\zeta_{e'}\in k$ where
$e'=\fn{lcm}\{\fn{ord}(\tau),\exp(H)\}$ and $\tau$ is an element
of $G$ whose image generates $G/H$. Then $k(G)$ is $k$-rational.
\end{theorem}

%---------------------------------------------S4
\section{Main results}

%-------------------t4.1
\begin{lemma} \label{t4.1}
Let $G$ be a group isomorphic to the group $(I)$ in Theorem
\ref{t2.7}. Suppose that $exp(G)=2^r s$ with $2\nmid s$. If $k$ is
an infinite field such that either $\fn{char}k=2$ or
$k(\zeta_{2^r})$ is a cyclic extension of $k$, then $k(G)$ is
retract $k$-rational.
\end{lemma}

\begin{proof} Note that $G= G_1 \rtimes G_2$ where $G_1\simeq C_m$, $G_2\simeq C_n$ with $\gcd\{m,n\}=1$.
By Theorem \ref{t3.1}, both $k(G_1)$ and $k(G_2)$ are retract
$k$-rational. Apply Theorem \ref{t1.11}. Done.
\end{proof}

%----------------------c4.2
\begin{theorem} \label{c4.2}
Let $G=N\rtimes G_0$ be a Frobenius group with kernel $N$ and
complement $G_0$. Let $exp(G)=2^r s$ where $2\nmid s$. Assume that
{\rm (i)} $G_0$ is a $Z$-group, and {\rm (ii)} $N$ is an abelian
group. If $k$ is an infinite field such that either $\fn{char}k=2$
or $k(\zeta_{2^r})$ is a cyclic extension of $k$, then $k(G)$ is
retract $k$-rational.
\end{theorem}

\begin{proof}
By Theorem \ref{t3.3}, $G_0$ is a group isomorphic to the group
(I) in Theorem \ref{t2.7}. Moreover, $\gcd\{|N|,|G_0|\}=1$ by
Theorem \ref{t2.1}.

Apply Theorem \ref{t3.1} and Lemma \ref{t4.1}. Both $k(N)$ and
$k(G_0)$ are retract $k$-rational. Thus $k(G)$ is also retract
$k$-rational by Theorem \ref{t1.11}.
\end{proof}

%---------------------l4.4
\begin{lemma} \label{l4.4}
Let $G$ be a group isomorphic to the group $(II)$ in Theorem
\ref{t2.7} with the parameters $m$, $n$, $r$, ... defined there
such that $n=2^uv$ $($where $u\ge 2, 2\nmid v)$. If $k$ is an
infinite satisfying that either $\fn{char}k=2$ or $\zeta_{2^u}\in
k$ $($when $\fn{char}k\ne 2)$, then $k(G)$ is retract
$k$-rational.
\end{lemma}

\begin{proof}
Define $G_1=\langle \sigma \rangle \subset G$, $G_2=\langle
\tau,\lambda \rangle$. Then $G\simeq G_1 \rtimes G_2$ with
$G_1\lhd G$.

We claim that $m$ is an odd integer. Otherwise, the subgroup
$\langle \sigma^{m/2}, \lambda^2 \rangle$ is isomorphic to $C_2
\times C_2$. This is impossible because $G$ is a $GZ$-group.

Thus $\gcd\{|G_1|,|G_2|\}=1$ and $k(G_1)$ is retract $k$-rational
by Theorem \ref{t3.1}. If $k(G_2)$ is retract $k$-rational, then
$k(G)$ is retract $k$-rational by Theorem \ref{t1.11}.

Define $\tau_1=\tau^{2^u}$, $\tau_2=\tau^v$, $H_1=\langle \tau_1\rangle$, $H_2=\langle \tau_2,\lambda\rangle$.
Then $G_2\simeq H_1 \rtimes H_2$ and $H_2$ is a 2-group,
while $2\nmid |H_1|$.
By Theorem \ref{t1.11}, if $k(H_2)$ is retract $k$-rational,
then $k(G_2)$ is also retract $k$-rational.

Note that $H_2=\langle \tau_2,\lambda\rangle$ is a 2-group with $\tau_2^{2^u}=\lambda^4=1$,
$\lambda\tau_2\lambda^{-1}=\tau_2^k$ and $\lambda^2=\tau_2^{2^{u-1}}$.

If $\fn{char}k\ne 2$ and $\zeta_{2^u}\in k$, by Theorem \ref{t3.4},
$k(H_2)$ is $k$-rational.
Hence $k(H_2)$ is retract $k$-rational.

If $\fn{char}k=2$, $k(H_2)$ is $k$-rational by Kuniyoshi's Theorem (see Theorem \ref{t3.8}).
Hence the result.
\end{proof}

%------------------------t4.5
\begin{theorem} \label{t4.5}
Define two finite groups $G_1$ and $G_2$ by
\[
G_1=\langle \tau,\lambda,\rho: \tau^{3^l}=\lambda^4=1,~ \lambda^2=\rho^2=(\lambda\rho)^2,~ \tau\lambda\tau^{-1}=\rho,~ \tau\rho\tau^{-1}=\lambda\rho\rangle
\]
where $l\ge 1$, and
\[
G_2=\langle G_1,\nu: \lambda^2=\nu^2,~ \nu\lambda\nu^{-1}=\rho\lambda,~\nu\rho\nu^{-1}=\rho^{-1},~\nu\tau\nu^{-1}=\tau^k\rangle
\]
where $k+1\equiv 0 \pmod{3}$ and $k^2\equiv 1 \pmod{3^l}$.

Let $k$ be a field with $\zeta_e \in k$ where exp$(G_2)=e$. Then
both $k(G_1)$ and $k(G_2)$ are $k$-rational.
\end{theorem}

Remark. Let $G_1$ and $G_2$ be the groups defined above. When
$l=1$ (recall $\tau^{3^l}=1$), it can be shown that $G_1 \simeq
\widetilde{A}_4 \simeq SL_2(\bm{F}_3)$ and $G_2 \simeq
\widehat{S}_4$ (see Definition \ref{d4.9}); the rationality
problem of $k(SL_2(\bm{F}_3))$ and $k(\widehat{S}_4)$ has been
solved in \cite{Ri} and \cite[Theorem 1.4]{KZ} respectively. When
$l \ge 2$ and $k$ is a field with $\fn{char}k=2$, we don't know
whether $k(G_1)$ and $k(G_2)$ are $k$-rational or not.

\begin{proof}
Note that $\langle \lambda,\rho \rangle \simeq \{\pm 1,\pm i,\pm
j,\pm k\}$ the quaternion group of order 8. Thus the $2$-Sylow
subgroups of $G_1$ and $G_2$ are quaternion group and the
generalized quaternion group of order $16$ respectively. It
follows that exp$(G_1)=4 \cdot 3^l$ and exp$(G_2)=8 \cdot 3^l$.

\medskip
\begin{Case}{1} The group $G_1$ and $\fn{char}k\ne 2,3$
with $\zeta_{3^l},\zeta_8\in k$.
\end{Case}

Write $\zeta=\zeta_{3^l}$. Define $\eta=\zeta_8$ with
$\eta^2=\sqrt{-1}$. Define a representation of $G_1$ by
\begin{align*}
\Phi:{} & G_1 \to GL_2(k) \\
& \lambda \mapsto \begin{pmatrix} \sqrt{-1} & 0 \\ 0 & -\sqrt{-1} \end{pmatrix}, \\
& \rho \mapsto \begin{pmatrix} 0 & -1 \\ 1 & 0 \end{pmatrix}, \\
& \tau \mapsto \frac{\zeta}{\sqrt{2}} \begin{pmatrix} -\eta & \eta \\ \eta^3 & \eta^3 \end{pmatrix}.
\end{align*}

$\Phi$ is a faithful irreducible representation of $G_1$. Let
$k\cdot x_1\oplus k\cdot x_2$ be its representation space. We can
embed $k\cdot x_1 \oplus k\cdot x_2$ into the regular
representation space $\oplus_{g\in G_1} k\cdot x_g$. Thus the
$G_1$-field $k(x_1,x_2)$ can be embedded into the $G_1$-field
$k(x_g:g\in G_1)$. Applying Theorem \ref{t3.6}, we get
$k(G_1)=k(x_g:g\in G_1)^{G_1}=k(x_1,x_2)^{G_1}(u_1,\ldots,u_t)$
where $t=|G_1|-2=8\cdot 3^l-2$ and $u_1,\ldots,u_t$ are elements
fixed by $G_1$.

Define $x=x_1/x_2$. Then $k(x_1,x_2)=k(x,x_2)$ and, for any $g\in
G_1$, $g\cdot x\in k(x)$, $g(x_2)=\alpha_g x_2$ for some $\alpha_g
\in k(x)$. Applying Theorem \ref{t3.5}, we get
$k(x,x_2)^{G_1}=k(x)^{G_1}(f)$ for some element $f$ fixed by
$G_1$. By L\"uroth's Theorem, $k(x)^{G_1}$ is $k$-rational. Hence
$k(G_1)$ is $k$-rational.

\medskip
\begin{Case}{2}
The group $G_2$ and $\fn{char}k\ne 2,3$ with $\zeta_{3^l}, \zeta_8\in k$.
\end{Case}

Since $3^l\mid k^2-1$ and $3\mid k+1$, it follows that $3^l\mid k+1$,
i.e. $\nu\tau\nu^{-1}=\tau^{-1}$.

\medskip
Step 1. As before, write $\zeta=\zeta_{3^l}$. Define a
representation of $G_2$ by
\begin{align*}
\Psi:{} & G_2 \to GL_4(k) \\
& \lambda \mapsto \left(\begin{array}{@{}cc;{2pt/2pt}cc@{}} \sqrt{-1} & 0 & & \\ 0 & -\sqrt{-1} & & \\ \hdashline[2pt/2pt]
& & 1 & 0 \\ & & 0 & 1 \end{array}\right) \\
& \rho \mapsto \left(\begin{array}{@{}cc;{2pt/2pt}cc@{}} 0 & -1 & & \\ 1 & 0 & & \\ \hdashline[2pt/2pt]
& & 1 & 0 \\ & & 0 & 1 \end{array}\right) \\
& \tau \mapsto \left(\begin{array}{@{}cc;{2pt/2pt}cc@{}} \frac{-1-\sqrt{-1}}{2} & \frac{1+\sqrt{-1}}{2} & & \\[1mm]
\frac{-1+\sqrt{-1}}{2} & \frac{-1+\sqrt{-1}}{2} & & \\[1mm] \hdashline[2pt/2pt]
& & \zeta & 0 \\ & & 0 & \zeta^{-1} \end{array}\right) \\
& \nu \mapsto \left(\begin{array}{@{}cc;{2pt/2pt}cc@{}} \frac{1}{\sqrt{-2}} & \frac{1}{\sqrt{-2}} & & \\[1mm]
\frac{1}{\sqrt{-2}} & \frac{-1}{\sqrt{-2}} & & \\[1mm] \hdashline[2pt/2pt]
& & 0 & 1 \\ & & 1 & 0 \end{array}\right)
\end{align*}

It is not difficult to see that $\Psi$ is a faithful
representation of $G_2$. Let $\oplus_{1\le i\le 4} k\cdot x_i$ be
its representation space. Thus we may embed $\oplus_{1\le i\le 4}
k\cdot x_i$ into the regular representation space $\oplus_{g\in
G_2} k\cdot x_g$. By the same method as in Case 1, we embed the
$G_2$-field $k(x_1,x_2,x_3,x_4)$ into the $G_2$-field $k(x_g:g\in
G_2)$ and apply Theorem \ref{t3.6}. We find that $k(G_2)$ is
rational over $k(x_1,x_2,x_3,x_4)^{G_2}$.

\medskip
Step 2.
Define $y_1=x_1/x_2$, $y_2=x_3/x_4$.
Then $k(x_1,x_2,x_3,x_4)=k(y_1,y_2,x_2,x_4)$ and,
for all $g\in G_2$, $g(y_1),g(y_2)\in k(y_1,y_2)$, $g(x_2)=\alpha_gx_2$, $g(x_4)=\beta_g x_4$ for some $\alpha_g,\beta_g\in k(y_1,y_2)$.
Applying Theorem \ref{t3.5} twice,
we find that $k(y_1,y_2,x_2,x_4)^{G_2}=k(y_1,y_2)^{G_2}(v_1,v_2)$ for some elements $v_1,v_2$ fixed by $G_2$.

Note that $\tau(x_1)= [(-1-\sqrt{-1})x_1/2]+
[(-1+\sqrt{-1})x_2/2]$, $\tau(x_2)= [(1+\sqrt{-1})x_1/2]+
[(-1+\sqrt{-1})x_2/2]$, etc. Thus the action of $G_2$ on
$k(y_1,y_2)$ is given by
\begin{align*}
\lambda &: y_1 \mapsto -y_1,~ y_2 \mapsto y_2, \\
\rho &: y_1\mapsto -1/y_1,~ y_2\mapsto y_2, \\
\tau &: y_1\mapsto (-y_1+\sqrt{-1})/(y_1+\sqrt{-1}),~ y_2\mapsto \zeta^2 y_2, \\
\nu &: y_1\mapsto (y_1+1)/(y_1-1),~ y_2\mapsto 1/y_2.
\end{align*}

Define $y_3=y_1^2+(1/y_1^2)$.
Then $k(y_1,y_2)^{\langle \lambda,\rho\rangle}=k(y_2,y_3)$ and
\[
\tau: y_3\mapsto (2y_3-12)/(y_3+2), \quad \nu: y_3\mapsto (2y_3+12)/(y_3-2).
\]

Note that $\tau^3(y_3)=y_3$. Define $y_4=y_2^{3^{l-1}}$. Then
$k(y_2,y_3)^{\langle \tau^3\rangle} =k(y_3,y_4)$.

The action of $\langle \tau,\nu\rangle$ on $k(y_3,y_4)$ is given by
\begin{align*}
\tau &: y_3\mapsto (2y_3-12)/(y_3+2),~ y_4\mapsto \omega^{\prime} y_4 \\
\nu &: y_3\mapsto (2y_3+12)/(y_3-2),~ y_4\mapsto 1/y_4
\end{align*}
where $\omega^{\prime}=\zeta^{2\cdot 3^{l-1}}$ and therefore
$\omega^{\prime}$ is a primitive 3rd root of unity.

\medskip
Step 3. Note that, for a $k$-automorphism $\sigma$ on $k(x)$ with
$\sigma(x)=(ax+b)/(cx+d)$ where $a,b,c,d\in k$ and $ad-bc\ne 0$,
if $y=(\alpha x+\beta)/(\gamma x+\delta)$ with
$\alpha,\beta,\gamma,\delta \in k$ and $\alpha\delta-\beta\gamma
\ne 0$, then $\sigma(y)=(Ay+B)/(Cy+D)$ where $A,B,C,D$ are defined
by
\[
\begin{pmatrix} \alpha & \beta \\ \gamma & \delta \end{pmatrix} \begin{pmatrix} a & b \\ c & d \end{pmatrix}
=\begin{pmatrix} A & B \\ C & D \end{pmatrix}.
\]

In particular, if
\[
\begin{pmatrix} \alpha & \beta \\ \gamma & \delta \end{pmatrix} \begin{pmatrix} a & b \\ c & d \end{pmatrix}
\begin{pmatrix} \alpha & \beta \\ \gamma & \delta \end{pmatrix}^{-1}
=\begin{pmatrix} tw & 0 \\ 0 & t \end{pmatrix}
\]
where $ad-bc \ne 0$, $\alpha\delta-\beta\gamma\ne 0$ and $\sigma(x)=(ax+b)/(cx+d)$,
then $\sigma(y)=wy$ if $y$ is defined as $y=(\alpha x+\beta)/(\gamma x+\delta)$.

\medskip
Step 4.
Return to the actions of $\tau$ and $\nu$ on $y_3$.
Since $\tau^3(y_3)=1$, it follows that the order of the matrix $\left(\begin{smallmatrix} 2 & -12 \\ 1 & 2 \end{smallmatrix}\right)\in PGL_2(k)$ is 3.
Regard this matrix as a $2\times 2$ matrix over $k$.
Its characteristic polynomial is $(X-2)^2+12=(X-2-2\sqrt{-3})(X-2+2\sqrt{-3})$.
Since $\zeta_3\in k$ (because $\zeta_{3^l} \in k$),
it follows the eigenvalues $2\pm \sqrt{-3} \in k$.
Hence this matrix can be diagonalized over $k$.
In other words, there is an invertible $2\times 2$ matrix $\left(\begin{smallmatrix} \alpha & \beta \\ \gamma & \delta \end{smallmatrix}\right) \in GL_2(k)$ such that
\[
\begin{pmatrix} \alpha & \beta \\ \gamma & \delta \end{pmatrix} \begin{pmatrix} 2 & -12 \\ 1 & 2 \end{pmatrix}
\begin{pmatrix} \alpha & \beta \\ \gamma & \delta \end{pmatrix}^{-1}
=\begin{pmatrix} a & 0 \\ 0 & b \end{pmatrix}.
\]

Since the order of $\left(\begin{smallmatrix} 2 & -12 \\ 1 & 2
\end{smallmatrix}\right) \in PGL_2(k)$ is 3, it is necessary that
$a/b=\omega$ is a primitive 3rd root of unity (where
$\omega^{\prime} =\omega$ or $\omega^2$).

Apply the result of Step 3. Define $y_5=(\alpha y_3+\beta)/(\gamma
y_3+\delta)$. Then $\tau(y_5)=\omega y_5$. Note that
$\nu(y_5)=(Ay_5+B)/(Cy_5+D)$ for some $A,B,C,D \in k$ with
$AD-BC\ne 0$.

Because $\nu\tau\nu^{-1}=\tau^{-1}$, we find $\tau\nu\tau=\nu$ and $\tau\nu\tau(y_5)=\nu(y_5)$.
Plugging in the formula $\nu(y_5)=(Ay_5+B)/(Cy_5+D)$,
we get the identity
\[
(\omega^2 Ay_5+\omega B)/(\omega Cy_5+D)=(Ay_5+B)/(Cy_5+D).
\]

The above identity implies that $A=D=0$.
In other words, $\nu(y_5)=d/y_5$ for some $d\in k\backslash \{0\}$.

In summary, we have $k(y_3,y_4)=k(y_4,y_5)$ and
\begin{align*}
\tau &: y_4\mapsto \omega^{\prime} y_4,~ y_5\mapsto \omega y_5, \\
\nu &: y_4\mapsto 1/y_4,~ y_5\mapsto d/y_5.
\end{align*}

By Theorem \ref{t3.9}, $k(y_4,y_5)^{\langle \tau,\nu\rangle}$ is $k$-rational.
Hence $k(G_2)$ is $k$-rational.
\end{proof}

%---------------------t4.6
\begin{theorem} \label{t4.6}
Let $G$ be a group isomorphic to the group $(III)$ in Theorem
\ref{t2.7} with the parameters $m$, $n$, $r$, ... defined there.
Write $n=3^l n'$ with $l\ge 1$ and $3\nmid n'$. If $k$ is an
infinite field satisfying that either $\fn{char}k=2$ or
$\zeta_{3^l},\zeta_8\in k$ $($when $\fn{char}k\ne 2,3)$, then
$k(G)$ is retract $k$-rational.
\end{theorem}

\begin{proof}
Step 1. Note that the parameter $m$ is an odd integer as in the
proof of Lemma \ref{l4.4}.

Define $H_1=\langle \sigma\rangle$, $H_2=\langle
\tau,\lambda,\rho\rangle$. Then $G\simeq H_1\rtimes H_2$, $H_1$ is
an abelian normal subgroup of $G$, and $\gcd\{|H_1|,|H_2|\}=1$.

Since $|H_1|=m$ is odd, $k(H_1)$ is retract $k$-rational by Theorem \ref{t3.1}.
We will prove that $k(H_2)$ is also retract $k$-rational.
Once it is proved, we may apply Theorem \ref{t1.11} to conclude that $k(G)$ is retract $k$-rational.

\medskip
Step 2. Note that $n=3^l\lambda'$ and $n$ is odd by assumption.
Define $\tau_1=\tau^{3^l}$ and $\tau_2=\tau^{n'}$. Then
$H_2=\langle \tau_1,\tau_2,\lambda,\rho\rangle$ with
$\tau_1^{n'}=\tau_2^{3^l}=1$ and
$\tau_1\lambda\tau_1^{-1}=\lambda$, $\tau_1\rho\tau_1^{-1}=\rho$.

Define $H_3=\langle \tau_1\rangle$, $H_4=\langle \tau_2,\lambda,\rho\rangle$.
Then $H_2\simeq H_3\rtimes H_4$, $H_3$ is an abelian normal subgroup of $H_2$,
and $\gcd\{|H_3|,|H_4|\}=1$.
$k(H_3)$ is retract $k$-rational by Theorem \ref{t3.1}.
We will prove that $k(H_4)$ is retract $k$-rational in the next step.
Once it is proved, we obtain that $k(H_2)$ is retract $k$-rational by Theorem \ref{t1.11}.

\medskip
Step 3. We will show that $k(H_4)$ is retract $k$-rational.

Note that $H_4$ is isomorphic to the group $G_1$ in Theorem \ref{t4.5}.
In fact, if $n'\equiv 1 \pmod{3}$, then $\tau_2\lambda\tau_2^{-1}=\rho$, $\tau_2\rho\tau_2^{-1}=\lambda\rho$;
thus $H_4$ is identical to $G_1$.
If $n'\equiv 2 \pmod{3}$, then $\tau_2\lambda\tau_2^{-1}=\lambda\rho$, $\tau_2\lambda\rho\tau_2^{-1}=\rho$;
replace the generators $\lambda$, $\rho$ by $\lambda$, $\lambda\rho$ and we find $H_4\simeq G_1$.

Case 1 \, $\fn{char}k\ne 2,3$ and $\zeta_{3^l},\zeta_8\in k$.

By Theorem \ref{t4.5} $k(G_1)$ is $k$-rational. Thus $k(H_4)$ is
$k$-rational. In particular, it is retract $k$-rational.

Case 2 \, $\fn{char}k=2$.

Consider the group extension $1\to \langle \lambda^2\rangle\to
H_4\to H_5\to 1$. Applying Theorem \ref{t3.7}, we find that
$k(H_4)$ is rational over $k(H_5)$.

Note that $H_5=\langle
\bar{\lambda},\bar{\rho},\bar{\tau}_2\rangle \simeq V\rtimes
C_{3^l}$ where $V=\langle\bar{\lambda},\bar{\rho}\rangle \simeq
C_2\times C_2$. Since both $k(V)$ and $k(C_{3^l})$ are retract
$k$-rational Theorem \ref{t3.1}, We find that $k(H_5)$ is also
retract $k$-rational.

Since retract rationality is stable under rational extension
\cite[Proposition 3.6; Ka4, Lemma 3.4]{Sa2}, we conclude that
$k(H_4)$ is also retract $k$-rational.
\end{proof}

%------------------------------t4.7
\begin{theorem} \label{t4.7}
Let $G$ be a group isomorphic to the group $(IV)$ in Theorem
\ref{t2.7} with the parameters $m$, $n$, $r$, ... defined there.
Write $n=3^l n'$ with $l\ge 1$ and $3\nmid n'$. If $k$ is an
infinite field with $\fn{char}k\ne 2,3$ and $\zeta_{3^l},
\zeta_8\in k$, then $k(G)$ is retract $k$-rational. In other
words, if exp$(G)=2^u3^lt$ where $u,l\ge 1$, $2\nmid t$, $3\nmid
t$, then $k(G)$ is retract $k$-rational provided that
$\zeta_{2^u}, \zeta_{3^l} \in k$.
\end{theorem}

\begin{proof}
The proof is similar to that of Theorem \ref{t4.6}.

Choose $H_1=\langle\sigma\rangle$, $H_2=\langle \tau,\lambda,\rho,\nu\rangle$.
Then $G\simeq H_1\rtimes H_2$.
To prove that $k(G)$ is retract $k$-rational,
it suffices to show that so is $k(H_2)$.

As $n=3^ln'$, define $\tau_1=\tau^{3^l}$, $\tau_2=\tau^{n'}$.
Define $H_2=\langle \tau_1\rangle$, $H_4=\langle
\tau_2,\lambda,\rho,\nu\rangle$. Then $H_2\simeq H_3\rtimes H_4$.
To prove that $k(H_2)$ is retract $k$-rational, it suffices to
show that so is $k(H_4)$.

Since $H_4$ is isomorphic to the group $G_2$ in Theorem \ref{t4.5}
and exp$(G_2)=8 \cdot 3^l$, it follows that $k(H_4)$ is
$k$-rational by Theorem \ref{t4.5}. Hence $k(H_4)$ is retract
$k$-rational.

For final assertion note that, if exp$(G)=2^u3^lt$, then
exp$(G_2)=2^u3^l$.
\end{proof}

%-----------------------------t4.8
\begin{theorem} \label{t4.8}
Let $G$ be a solvable $GZ$-group of exponent $2^u3^lt$ where
$u,l\ge 0$, $2\nmid t$, $3\nmid t$. If $k$ is an infinite field
such that $\fn{char}k\ne 2,3$, $\zeta_{3^l}\in k$ and
$\zeta_{2^{u'}}$ $($where $u'=\max\{3,u\})$, then $k(G)$ is
retract $k$-rational.
\end{theorem}

\begin{proof}
By Theorem \ref{t2.7}, $G$ is isomorphic to the groups (I)--(IV).
Apply Lemma \ref{t4.1}, Lemma \ref{l4.4}, Theorem \ref{t4.6} and
Theorem \ref{t4.7}.
\end{proof}

%--------------------------d4.9
\begin{defn} \label{d4.9}
For $n \ge 4$, there are two inequivalent non-split central
extensions of $S_n$ by $\bm{Z}/2\bm{Z}$. We follow the notations
of Serre \cite[pages 58, 88, 90]{GMS}. The non-split central
extension $1\to\{\pm 1\}\to \widehat{S}_n \to S_n\to 1$ defines a
double cover $\widehat{S}_n$ of $S_n$ in which the transposition
and the product of two disjoint transpositions in $S_n$ lift to
elements of order 4 of $\widehat{S}_n$. The non-split central
extension $1\to \{\pm 1\} \to \widetilde{S}_n \to S_n\to 1$
defines a double cover $\widetilde{S}_n$ of $S_n$ in which a
transposition in $S_n$ lifts to an element of order 2 of
$\widetilde{S}_n$, but a product of two disjoint transpositions
lifts to an element of order 4. The non-split central extension
$1\to \{\pm 1\}\to \widetilde{A}_n \to A_n\to 1$ defines the
(unique) double cover $\widetilde{A}_n$ of $A_n$ \cite[page
88]{Se}.
\end{defn}

%---------------------------l4.10
\begin{lemma} \label{l4.10}
{\rm (1)} $SL_2(\bm{F}_5)\simeq \widetilde{A}_5$.

{\rm (2)} Let $G$ be the group (II) in Theorem \ref{t2.8} with
$p=5$. Let $G_+$ be the subgroup of $G$ defined by $G_+=\langle
\lambda,L\rangle$. Then $G_+\simeq \widehat{S_5}$.
\end{lemma}

\begin{proof}
Step 1. Note that $1\to\{\pm 1\}\to \widetilde{A}_5 \to A_5\to 1$
is the unique non-split extension of $A_5$ by $\{\pm 1\}$
\cite[page 94; Se, page 88]{Kar}.

Since $PSL_2(\bm{F}_5)$ is a simple group of order 60,
we find that $PSL_2(\bm{F}_5)\simeq A_5$.
Hence the group extension $1\to\{\pm 1\}\to SL_2(\bm{F}_5)\to PSL_2(\bm{F}_5)\simeq A_5 \to 1$ gives a Schur covering group of $A_5$.
By the uniqueness, we conclude that $SL_2(\bm{F}_5)\simeq \widetilde{A}_5$.

\medskip
Step 2. The binary icosahedral group $H$ is defined in \cite[page
93]{Sp} as follows. For a field $k$ with $\fn{char}k\ne 2,5$ and
$\zeta_5\in k$, $H$ is the subgroup of $GL_2(k)$ defined by
$H=\langle a,b,c\rangle$ where
\[
a=-\begin{pmatrix} \zeta^3 & 0 \\ 0 & \zeta^2 \end{pmatrix}, \quad
b=\begin{pmatrix} 0 & 1 \\ -1 & 0 \end{pmatrix}, \quad
c=\frac{1}{\zeta^2-\zeta^{-2}} \begin{pmatrix} \zeta+\zeta^{-1} & 1 \\ 1 & -(\zeta+\zeta^{-1}) \end{pmatrix}
\]
with $\zeta=\zeta_5$.

Also by \cite[page 93]{Sp}, $H$ may be presented by exhibiting
generators and relations as $H=\langle a,b,c\rangle$ where
$Z(H)=\langle \varepsilon\rangle$ with $\varepsilon=a^5$ and
\[
\varepsilon^2=1,~ a^5=b^2=c^2=\varepsilon,~ bab^{-1}=a^{-1},~ bcb^{-1}=\varepsilon c,~ cac=acba,~ ca^2c=a^{-2}ca^{-2}.
\]

It follows that $H=\{a^i,ba^i,a^ica^j,a^icba^j:0\le i\le 9,0\le
j\le 4\}$ is a group of order 120.

Note that the group homomorphism
\begin{align*}
\pi:{} & H \to A_5 \\
& a \mapsto (1~2~3~4~5) \\
& b \mapsto (1~4) (2~3) \\
& c \mapsto (1~3) (2~4)
\end{align*}
defines a central extension $1\to \{\pm 1\} \to H \xrightarrow{\pi} A_5 \to 1$.
Hence $H\simeq \widetilde{A}_5$.

\medskip
Step 3. By Step 1 and Step 2, $SL_2(\bm{F}_5)\simeq
\widetilde{A}_5 \simeq H$. We will exhibit an isomorphism from
$SL_2(\bm{F}_5)$ onto $H$.

Define $A,B,C \in SL_2(\bm{F}_5)$ by
\[
A=\begin{pmatrix} -1 & 1 \\ 0 & -1 \end{pmatrix}, \quad
B=\begin{pmatrix} 2 & 1 \\ 0 & -2 \end{pmatrix}, \quad
C=\begin{pmatrix} 2 & 0 \\ 2 & -2 \end{pmatrix}.
\]

It is not difficult to verify that $SL_2(\bm{F}_5)=\langle A,B,C\rangle$ and the group homomorphism $\varphi:SL_2(\bm{F}_5) \to H$ defined by $\varphi(A)=a$,
$\varphi(B)=b$, $\varphi(C)=c$ is an isomorphism.

\medskip
Step 4. We will study the automorphism $\theta:SL_2(\bm{F}_5)\to
SL_2(\bm{F}_5)$ defined in Theorem \ref{t2.8}. Choose $\omega
=2\in \bm{F}_5$. Then $\theta$ is given by
\[
\rho\mapsto \begin{pmatrix} 0 & -1 \\ 2 & 0 \end{pmatrix} \rho
\begin{pmatrix} 0 & -1 \\ 2 & 0 \end{pmatrix}^{-1}
\]
for any $\rho \in SL_2(\bm{F}_5)$.
It follows that $\theta(A)=-A^3CA$, $\theta(B)=-C$, $\theta(C)=-B$.

\medskip
Step 5.
Now we turn to the group $G_+=\langle \lambda,L\rangle$ of this theorem.

Because $L\simeq SL_2(\bm{F}_5)\simeq H$,
we may use the presentation of $H$ in Step 2 for a presentation of $L$,
i.e.\ we write $L=\langle a,b,c \rangle$ with the relations given there.
It follows that $\theta(a)=\varepsilon a^3ca$, $\theta(b)=\varepsilon c$, $\theta(c)=\varepsilon b$ by Step 4.
Define a group $G_-$ by $1\to \{1,\varepsilon\} \to G_+ \to G_-\to 1$.
Then $G_-=\langle \bar{a},\bar{b},\bar{c},\bar{\lambda}\rangle$ with the relations induced from $a$, $b$, $c$ and the relation
$\bar{\lambda}\bar{g}\bar{\lambda}^{-1}=\theta(\bar{g})$ for any $\bar{g}\in \langle \bar{a},\bar{b},\bar{c}\rangle$.

It is easy to check the following map is a well-defined group homomorphism
\begin{align*}
\psi :{} & G_- \to S_5 \\
& \bar{a} \mapsto (1~2~3~4~5) \\
& \bar{b} \mapsto (1~4)(2~3) \\
& \bar{c} \mapsto (1~3)(2~4) \\
& \bar{\lambda} \mapsto (3~4).
\end{align*}

Moreover, $\psi$ is an onto map.
Since $|G_-|=120$, it follows that $\psi$ is an isomorphism.

Thus we get a group extension $1\to \{1,\varepsilon\}\to G_+\to
G_- \simeq S_5\to 1$. Since $\lambda\in G_+$ is mapped to
$(3~4)\in S_5$ and $\lambda$ is an element of order $4$, it
follows that $G_+\simeq \widehat{S}_5$ by Definition \ref{d4.9}.
\end{proof}

%-------------------------t4.11
\begin{theorem} \label{t4.11}
Let $G=G_1\times G_2$ be a $GZ$-group where $G_1$ is a $Z$-group,
$G_2\simeq SL_2(\bm{F}_5)$, i.e.\ $G$ is the group $(I)$ in
Theorem \ref{t2.8} with $p=5$.

Assume that $k$ is a field satisfying at least one of the
following conditions : $\fn{char}k=0$, $\fn{char}k=2$, or
$\fn{char} k=l > 0$ with $l \equiv \pm 1 \pmod{5}$. Then $k(G_2)$
is $k$-rational. If it is assumed furthermore that $k$ is an
infinite field, then $k(G)$ is retract $k$-rational.
\end{theorem}

\begin{proof}
Consider $k(G_2)$ first.

Note that $G_2\simeq SL_2(\bm{F}_5)\simeq \widetilde{A}_5$ by
Lemma \ref{l4.10}.

Suppose $\fn{char}k=0$. By \cite[Theorem 14]{Pl},
$\bm{Q}(\widetilde{A}_5)$ is rational. Thus for any field $k$ with
$\fn{char}k=0$, $k(\widetilde{A}_5)$ is also $k$-rational.

If $\fn{char}k=2$, consider the central extension $1\to \{\pm 1\}
\to \widetilde{A}_5\to A_5\to 1$. Applying Theorem \ref{t3.8}, we
find that $k(\widetilde{A}_5)$ is rational over $k(A_5)$. By
Maeda's Theorem \cite{Ma}, $k(A_5)$ is $k$-rational for any field.
Hence $k(\widetilde{A}_5)$ is $k$-rational.

Finally consider the case when $k$ is field with $\fn{char} k=l >
0$ with $l^2 \equiv 1 \pmod{5}$.

By \cite[page 500]{Hu}, there is a faithful representation $\Phi:
\widetilde{A}_5 \simeq SL_2(\bm{F}_5) \to GL_2(k)$ (in case $l
\equiv 1 \pmod{5}$, we may use the representation in Step 2 of the
proof of Lemma \ref{l4.10}). Let $k\cdot x_1\oplus k\cdot x_2$ be
its representation space. We can embed $k\cdot x_1\oplus k\cdot
x_2$ into the regular representation space $\oplus_{g\in
\widetilde{A}_5} k\cdot x_g$. Thus the $\widetilde{A}_5$-field
$k(x_1,x_2)$ can be embedded into $k(x_g: g\in \widetilde{A}_5)$.
Applying Theorem \ref{t3.6}, we find that
$k(\widetilde{A}_5)=k(x_g:g\in \widetilde{A}_5)^{\widetilde{A}_5}$
is rational over $k(x_1,x_2)^{\widetilde{A}_5}$. Write
$x=x_1/x_2$. We get $k(x_1,x_2)=k(x,x_2)$. The proof is similar to
the last paragraph of Case 1 in the proof of Theorem \ref{t4.5}.
We get $k(x_1,x_2)^{A_5}=k(x,x_2)^{A_5}$ is $k$-rational.

Now consider $k(G)$. Note that $\gcd\{|G_1|,|G_2|\}=1$. Thus
$|G_1|$ is an odd integer. $k(G_1)$ is retract $k$-rational by
Lemma \ref{t4.1} (with the aid of Theorem \ref{t3.3}). Using the
result that $k(G_2)$ is $k$-rational, we conclude that $k(G)$ is
retract $k$-rational by Theorem \ref{t3.2}.
\end{proof}

%---------------------------t4.12
\begin{theorem} \label{t4.12}
Let $G$ be the group $(II)$ in Theorem \ref{t2.8} with $p=5$.
Define a subgroup $G_+$ of $G$ by $G_+=\langle \lambda,L\rangle$.
If $k$ is a field with $\fn{char}k=2$ or $\fn{char}k=0$ such that
$k(\zeta_8)$ is a cyclic extension of $k$, then $k(G_+)$ is
$k$-rational. If it is assumed furthermore that $k$ is an infinite
field, then $k(G)$ is retract rational.
\end{theorem}

\begin{proof}
By Lemma \ref{l4.10}, $G_+\simeq \widehat{S}_5$. We will show that
$k(G_+)$ is $k$-rational.

If $\fn{char}k=0$ and $k(\zeta_8)$ is cyclic over $k$, then
$k(G_+)$ is $k$-rational by \cite[Theorem 1.4]{KZ}.

If $\fn{char}k=2$, consider the group extension $1\to \{\pm 1\}\to
\widehat{S}_5 \simeq G_+\to S_5\to 1$. By Theorem \ref{t3.8},
$k(\widehat{S}_5)$ is rational over $k(S_5)$. But $k(S_5)$ is
$k$-rational. Hence so is $k(\widehat{S}_5)$.

Now we will show that $k(G)$ is retract $k$-rational.
The proof is similar to that of Theorem \ref{t4.6}.

Define $H_1=\langle \sigma\rangle$, $H_2=\langle \tau,\lambda,L\rangle$.
We obtain $G\simeq H_1\rtimes H_2$.
Since $k(H_1)$ is retract $k$-rational by Theorem \ref{t3.1},
we find that $k(G)$ is retract $k$-rational if and only if $k(H_2)$ is retract $k$-rational by Theorem \ref{t1.11}.

Define $H_3=\langle \tau \rangle$ and $G_+=\langle
\lambda,L\rangle$. Then $H_2\simeq H_3\rtimes G_+$. Using the same
arguments, we find that $k(H_2)$ is retract $k$-rational if and
only so is $k(G_+)$.

Under the assumption that (i) $\fn{char}k=0$ with $k(\zeta_8)$
cyclic over $k$, or (ii) $\fn{char}k=2$ and $k$ is infinite, it is
clear that $k(G_+)$ is retract $k$-rational because it is
$k$-rational.
\end{proof}

%----------------------------t4.13
\begin{example} \label{ex4.13}
Let $G$ and $G_+$ be the same as in Theorem \ref{t4.12}, and $k$
be a field with $\fn{char} k=0$. Checking the proof of Theorem
\ref{t4.12}, we find that $k(G)$ is retract $k$-rational if and
only if so is $k(G_+)$.

By Serre's Theorem, $\bm{Q}(\widehat{S}_5)$ is not retract
$\bm{Q}$-rational \cite[Example 33.27, page 90; KZ, Theorem
1.2]{GMS}. In particular, $\bm{Q}(G)$ is not retract
$\bm{Q}$-rational.

Let $\widehat{G}$ be a Frobenius group with Frobenius complement
$G$ defined above. We claim that $\bm{Q}(\widehat{G})$ is not
retract $\bm{Q}$-rational.

Suppose that $\bm{Q}(\widehat{G})$ is retract $\bm{Q}$-rational.
By Theorem \ref{t1.11}, we find that $\bm{Q}(G)$ would be retract
$\bm{Q}$-rational, which is a contradiction.

This example shows that the assumption $k(\zeta_8)/k$ being cyclic
is crucial in Theorem \ref{t1.8}.

By the same method, it is possible to find a solvable Frobenius
group $\widehat{G}$ whose Frobenius complement is the group
$(III)$ or $(IV)$ in Theorem \ref{t2.7} such that
$\bm{Q}(\widehat{G})$ is not retract $\bm{Q}$-rational. For, the
$2$-Sylow subgroup of $\widehat{G}$ is a generalized quaternion
group of order $\ge 16$. Apply Serre's Theorem that
$\bm{Q}(G_{16})$ is not retract $\bm{Q}$-rational (where $G_{16}$
is the generalized quaternion group of order $16$) \cite[Theorem
34.7, page 92; Ka4, Section 1]{GMS}. The details are omitted.
\end{example}

\bigskip
Now we give the proof for results in Section 1.

\begin{proof}[Proof of Theorem \ref{t1.9}]
Write $G=N\rtimes G_0$ where $N$ is abelian and $G_0$ is solvable.
By Theorem \ref{t1.6} and Theorem \ref{t2.7}, $G_0$ is isomorphic
to the groups (I)--(IV) listed in Theorem \ref{t2.7}. Thus we may
apply Theorem \ref{t4.8} to show that $k(G_0)$ is retract
$k$-rational.

As to $k(N)$, $k(N)$ is retract $k$-rational by Theorem \ref{t3.1}
if $|N|$ is odd. If $|N|$ is even and the exponent of $N$ is $2^u
n'$ (where $2 \nmid n'$), then the exponent of $G$ is $2^u m$
(where $2 \nmid m$). Since $\zeta_{2^{u'}} \in k$, it follows that
$k(N)$ is also retract $k$-rational by Theorem \ref{t3.1}.

Applying Theorem \ref{t1.11}, we find the $k(G)$ is retract $k$-rational.
\end{proof}

\begin{proof}[Proof of Theorem \ref{t1.8}]
Write $G=N\rtimes G_0$ where $G_0$ is non-solvable. By Theorem
\ref{t1.6}, Theorem \ref{t2.8} and Theorem \ref{t2.9}, $G_0$ is
isomorphic to the groups (I) or (II) listed in Theorem \ref{t2.8}
with $p=5$. Applying Theorem \ref{t4.11} and Theorem \ref{t4.12},
we find that $k(G_0)$ is retract $k$-rational.

By Theorem \ref{t2.1}, $N$ is of odd order and is abelian, because
$|G_0|$ is even and $\gcd\{|N|,|G_0|\}=1$. By Theorem \ref{t3.1},
$k(N)$ is retract $k$-rational.

By Theorem \ref{t1.11}, we find that $k(G)$ is retract
$k$-rational.
\end{proof}

\begin{proof}[Proof of Theorem \ref{t1.14}]
The proof is similar to that of Theorem \ref{t1.8}. This time we
apply only Theorem \ref{t4.11} because we are working on the
groups $(I)$ in Theorem \ref{t2.8}. Hence the result.
\end{proof}

%-------------------------------------S5
\section{Remarks about Bogomolov multipliers}

First of all, let us recall the notions of unramified Brauer
groups and Bogomolov multipliers.

Let $k\subset K$ be an extension of fields. The unramified Brauer
group of $K$ over $k$, denoted by $\fn{Br}_{v,k}(K)$ was defined
as $\fn{Br}_{v,k}(K)=\bigcap_R \fn{Image} \{ \fn{Br}(R)\to
\fn{Br}(K)\}$ where $\fn{Br}(R)\to \fn{Br}(K)$ is the natural map
of Brauer groups and $R$ runs over all the discrete valuation
rings $R$ such that $k\subset R\subset K$ and $K$ is the quotient
field of $R$.

By \cite{Sa3}, $\fn{Br}_{v,k}(K)$ is an obstruction for $K$ to be
$k$-rational. In particular, if $k$ is an algebraically closed
field and $K$ is retract $k$-rational, then $\fn{Br}_{v,k}(K)=0$.
The following result shows that $\fn{Br}_{v,k}(k(G))$ depends only
on the group $G$.

%---------------------t1.4
\begin{theorem}[{Bogomolov, Saltman \cite[Theorem 12]{Bo,Sa4}}] \label{t1.4}
Let $G$ be a finite group, $k$ be an algebraically closed field
with $\gcd \{|G|,\fn{char}k\}=1$.  Then $\fn{Br}_{v,k}(k(G))$ is
isomorphic to the group $B_0(G)$ defined by
\[
B_0(G)=\bigcap_A \fn{Ker} \{\fn{res}_G^A: H^2(G,\bm{Q}/\bm{Z})\to
H^2(A,\bm{Q}/\bm{Z})\}\
\]
where $A$ runs over all the bicyclic subgroups of $G$ $($a group
$A$ is called bicyclic if $A$ is either a cyclic group or a direct
product of two cyclic groups$)$.
\end{theorem}

Thus, if $\bm{C}(G)$ is retract $\bm{C}$-rational, then
$B_0(G)=0$.

By Theorem \ref{t1.9} and Theorem \ref{t1.8}, $\bm{C}(G)$ is
retract $\bm{C}$-rational for any Frobenius group $G$ with abelian
Frobenius kernel. In particular, $B_0(G)=0$, a phenomenon observed
by Moravec \cite[Corollary 6.6]{Mo}. In fact, our result of the
retract rationality of $\bm{C}(G)$ also implies
$H_{nr,\bm{C}}^q(\bm{C}(G), \bm{Q}/\bm{Z})=0$ for all $q \ge 3$
where $H_{nr,\bm{C}}^q(\bm{C}(G), \bm{Q}/\bm{Z})$ are the higher
unramified cohomology groups defined by Colliot-Th\'el\`ene and
Ojanguren \cite{CTO}.

\medskip
We remark that the assumption of abelian Frobenius kernels can be
waived by the following lemma.

%------------------------l5.1
\begin{lemma} \label{l5.1}
Let $G$ be a Frobenius group with Frobenius kernel $N$.
If $B_0(N_p)=0$ for all Sylow subgroups $N_p$ of $N$, then $B_0(G)=0$.
\end{lemma}

\begin{proof}
Let $\fn{res}:B_0(G)\to B_0(N)$ be the map induced from the
restriction map $\fn{res} : H^2(G, \bm{Q}/\bm{Z}) \to H^2(N,
\bm{Q}/\bm{Z})$. It is shown in the proof of \cite[Corollary
6.6]{Mo} that $\fn{res}:B_0(G)\to B_0(N)$ is injective. Since
$B_0(N)\to \oplus_p B_0(N_p)$ is also injective where $p$ runs
over all the prime divisors of $|N|$, thus $B_0(G)\to \oplus_p
B_0(N_p)$ is injective. Done.
\end{proof}

From Lemma \ref{l5.1}, it is important to know the situations for
which $B_0(H)=0$ (where $H$ is a $p$-group). The following example
is a partial answer to it.

%-------------------------ex5.2
\begin{example} \label{ex5.2}
If $H$ is a $p$-group, we list several sufficient conditions to
ensure that $B_0(H)=0$.

First, if $H$ is abelian, then $\bm{C}(H)$ is $\bm{C}$-rational by
Theorem \ref{t3.7}. Hence $B_0(H)=0$.

For any $p$-group $H$, if $H$ is meta-cyclic or $H$ contains a cyclic subgroup of index $p^2$,
then $\bm{C}(H)$ is $\bm{C}$-rational by \cite{Ka1,Ka3}.
Hence $B_0(H)=0$.

If $H$ has an abelian normal subgroup $H_0$ such that $H/H_0$ is a
cyclic group, then $B_0(H)=0$ by \cite[Theorem 5.10]{Bo,Ka4}.

If $H=H_0\rtimes B$ where $H_0$ is abelian normal in $H$ and $B$ is bicyclic,
then $B_0(H)=0$ by \cite[Theorem 4.2]{Ba,HKK}.

If $H$ is a 2-group of order $\le 32$, then $\bm{C}(H)$ is $\bm{C}$-rational by \cite{CK,CHKP}.
Hence $B_0(H)=0$.
If $H$ is a group of order 64 and $H$ is not in the 13rd isoclinism family,
then $B_0(H)=0$ \cite[Theorem 1.14]{CHKK,HKK}.
(For the numbering of the isoclinism families, see \cite[Section 1]{HKK} for details.)

If $p$ is an odd prime number and $H$ is a $p$-group of order $\le
p^4$, then $\bm{C}(H)$ is $\bm{C}$-rational by \cite{CK}. Thus
$B_0(H)=0$. If $H$ is a group of order $p^5$ and $H$ is not in the
10th isoclinism family, then $B_0(H)=0$ \cite{HKK}.
\end{example}

\newpage
%----------------------------------------References
\renewcommand{\refname}{\centering{References}}

\end{document}